\documentclass[10pt]{article}
\usepackage{amsmath,amssymb,amsthm,amscd}
\numberwithin{equation}{section}

\def\cE{{\mathcal E}}

\def\cK{{\mathcal K}}
\def\cL{{\mathcal L}}

\def\cO{{\mathcal O}}

\def\cX{{\mathcal X}}

\def\bB{{\mathbf B}}
\def\bG{{\mathbf G}}
\def\bL{{\mathbf L}}

\def\bD{{\mathbf D}}
\def\bE{{\mathbf E}}

\def\bF{{\mathbf F}}

\def\bT{{\mathbf T}}
\def\bF{{\mathbf F}}
\def\bJ{{\mathbf J}}
\def\bK{{\mathbf K}}

\def\bU{{\mathbf U}}
\def\bV{{\mathbf V}}

\def\bU{{\mathbf U}}
\def\bW{{\mathbf W}}

\def\bM{{\mathbf M}}

\def\SL{{\mathbf S}{\mathbf L}}
\def\gl{{\mathbf g}{\mathbf l}}

\def\QQ{{\mathbb Q}}

\def\CC{{\mathbb C}}

\newtheorem{prop}{Proposition}[section]
\newtheorem{theo}[prop]{Theorem}
\newtheorem{lemm}[prop]{Lemma}

\newtheorem{rema}[prop]{Remark}

\newtheorem{defi}[prop]{Definition}

\def\begeq{\begin{equation}}
\def\endeq{\end{equation}}

\def\lab{\ }
\def\lab{\label}

\title{K-stability implies CM-stability}
\author{Gang Tian\thanks{Supported partially by
grants from NSF and NSFC}
\\Beijing University and Princeton University
}
\date{}

\begin{document}

\maketitle
\tableofcontents

\section{Introduction}
In this paper, we prove that any polarized K-stable manifold is CM-stable.
%\footnote{ Both notions of the stability originated in\cite{tian97}.}
This has been known to me for quite a while, in fact, the case for Fano manifolds already appeared in \cite{tian12} and
our arguments for the proof here will follow the approach there.

Let $M$ be a projective manifold polarized by an ample line bundle $L$.
By the Kodaira embedding theorem, for $\ell$
sufficiently large, a basis of $H^0(M, L^\ell)$ gives an embedding
$\phi_\ell: M \mapsto \CC P^{N}$, where $N=\dim_\CC
H^0(M,L^\ell)-1$. Any other basis gives an embedding of the form
$\sigma\cdot \phi_\ell$, where $\sigma\in \bG=\SL(N+1, \CC)$. We fix such an embedding.

Let us recall the CM-stability which originated in \cite{tian97}.
It can be defined in terms of Mabuchi's K-energy:
\begin{equation}
\label{eq:cm-1}
\bM_{\omega_0}(\varphi)\,=\,-\frac{1}{V}\,\int_0^1 \int_M \varphi \,({\rm Ric}(\omega_{t\varphi}) - \mu \,\omega_{t\varphi} )\wedge \omega_{t\varphi}^{n-1} \wedge dt,
\end{equation}
where $\omega_0$ is a K\"aher metric with K\"ahler class $2\pi c_1(L)$ and
\begin{equation}\label{eq:mu}
\omega_\varphi\,= \,\omega_0 \,+\,\sqrt{-1}\,\partial\bar\partial \,\varphi\,
~~{\rm and}~~ \,\mu\,=\,\frac{c_1(M)\cdot c_1(L)^{n-1}}{c_1(L)^n}.
\end{equation}

Given an embedding $M\subset \CC P^N$ by $K_M^{-\ell}$, we have an induced function on $\bG\,=\,\SL(N+1,\CC)$ which acts on $\CC P^N$:
\begin{equation}\label{eq:cm-2}
\bF(\sigma)\,=\,\bM_{\omega_0}(\psi_\sigma),
\end{equation}
where $\psi_\sigma$ is defined by
\begin{equation}\label{eq:cm-3}
\frac{1}{\ell}\,\sigma^*\omega_{FS}\,=\,\omega_0\,+\, \sqrt{-1}\,\partial\bar\partial\, \psi_\sigma.
\end{equation}
Note that $\bF(\sigma)$ is well-defined since $\psi_\sigma$ is unique modulo addition of constants. Similarly, we can define
$\bJ$ on $\bG$ by
\begin{equation}\label{eq:cm-3'}
\bJ(\sigma)\, =\, \bJ_{\omega_0}(\psi_\sigma),
\end{equation}
where
\begin{equation}
\label{eq:cm-3''}
J_{\omega_0}(\varphi)\,=\, \sum_{i=0}^{n-1} \frac{i+1}{n+1} \int_M
\sqrt{-1} \, \partial \varphi\wedge \bar{\partial} \varphi \wedge
\omega_0^i\wedge \omega_\varphi^{n-i-1}.
\end{equation}
\begin{defi}
\label{defi:cm-1}
We call $M$ CM-semistable with respect to $L^{\ell}$ if $\bF$ is bounded from below and
CM-stable with respect to $L^{\ell}$ if $\bF$ bounded from below and is proper modulo $\bJ$, i.e.,
for any sequence $\sigma_i\in \bG$,
\begin{equation}\label{eq:cm-defi}
\bF(\sigma_i) \rightarrow\infty~{\rm whenever}~\inf_{\tau \in Aut_0 (M,L)}\bJ(\sigma_i\tau) \rightarrow \infty,
\end{equation}
where $Aut_0(M,L)$ denotes the identity component of the automorphism group of $(M,L)$.
If $Aut_0(M,L)$ is trivial, then \eqref{eq:cm-defi} simply means that $\bF(\sigma_i) \rightarrow\infty$ whenever $\bJ(\sigma_i) \rightarrow \infty$.

We say $(M,L)$ CM-stable (resp. CM-semistable) if $M$ is CM-stable (resp. CM-semistable) with respect to $L^{\ell}$ for all sufficiently large $\ell$.
\end{defi}
\begin{rema}
\label{rema:cm}
In \cite{tian97}, the CM-stability is defined in terms of
the orbit of a lifting of $M$ in certain determinant line bundle, referred as the CM-polarization.
Theorem 8.9 in \cite{tian97} states that such an algebraic formulation is
equivalent to the one in Definition \ref{defi:cm-1}. 
\end{rema}

The CM-stability of $(M, L)$ is directly related to the existence of K\"ahler metrics with constant scalar curvature and K\"ahler class $c_1(L)$.
When $M$ is a Fano manifold polarized by
the anti-canonical bundle $K_M^{-1}$, it follows from \cite{tian09} and the partial $C^0$-estimate that
{\it $M$ admits a K\"ahler-Einstein metric whenever it is CM-stable} (see \cite{tian12}). In general, we had proposed a similar program towards
the YTD conjecture: If $(M,L)$ is K-stable, then there is a K\"ahler metric with constant scalar curvature and K\"ahler class $2\pi c_1(L)$.

We say $\bM_{\omega_0}$ is proper on the space $P(M, \omega_0)=\{\varphi\in C^\infty\,|\,\omega_\varphi > 0\}$ if there is a function
$f$ bounded from below such that $\lim _{t\to\infty} f(t) = \infty$ and
\begin{equation}\label{eq:cm-defi}
\bM_{\omega_0}(\varphi) \,\ge\,\inf_{\tau \in Aut_0 (M,L)} f(\bJ_{\omega_0}(\varphi_\tau)),~~~\forall \varphi \in P(M,\omega_0),
\end{equation}
where $\varphi_\tau$ is given by $\tau^*\omega_\varphi = \omega_0+\sqrt{-1}\,\partial\bar\partial \varphi_\tau$.
It was conjectured (see [tian98]\footnote{In \cite{tian98}, we define the properness in the case that $Aut_0(M,L)$ is trivial. Also one can easily 
show that the property of properness is independent of the choice of $\omega_0$.})
that $M$ admits a K\"ahler metric of constant scalar curvature and K\"ahler class $2\pi c_1(L)$ if $\bM_{\omega_0}$ is proper on $P(M,\omega_0)$. 
We also conjecture that a version of partial $C^0$-estimate holds for K\"ahler metrics with K\"ahler class $2\pi c_1(L)$. If these conjectures can be verified,
we can solve the YTD conjecture.

\vskip 0.1in
Our main result of this paper is the following:

\begin{theo}
\label{th:main}
Let $(M,L)$ be a polarized projective manifold which is K-stable. Then $M$ is CM-stable with respect to any $L^\ell$ which is
very ample. In particular, $(M,L)$ is CM-stable.
\end{theo}
We refer the readers to \cite{tian97}, \cite{donaldson02} and Subsection 4.1 of \cite{tian09} for the definition of the K-stability.
If $M$ is a K-stable Fano manifold, one can deduce from Theorem \ref{th:main} the existence
of a K\"ahler-Einstein metric on $M$. This is exactly the second approach in \cite{tian97} to 
complete the proof of the YTD conjecture for Fano manifolds.

Using the asymptotics of the K-energy in Lemma \ref{lemm:k-1}, one can easily show the converse: The CM-stability implies the K-stability.

The rest of this paper is devoted to the proof of Theorem \ref{th:main}.

\vskip 0.1in

{\bf Acknowledgement}: I would like to thank Chi Li and Sean Paul for useful comments on 
improving presentation of this paper.

\vskip 0.1in

\section{Asymptotics of the K-energy}

In this section, we recall a result which relates the Futaki invariant to the asymptotic expansion of the K-energy.
Let $\bG_0$ be an algebraic subgroup $\bG_0=\{\sigma(t)\}_{t\in \CC ^*}\subset \bG$, then
there is a unique limiting cycle
\begin{equation}\label{eq:k-1}
M_0\,=\,\lim_{t\to 0} \sigma(t)(M)\subset \CC P^N.
\end{equation}
If $\bG_0$ acts on $M_0$ non-trivially, one can associate the generalized Futaki invariant $f_{M_0, L_0}(\bG_0)$ for $M_0$,
where $L_0^\ell  = \cO(1)|_{M_0}$.
This invariant was defined by Ding-Tian for normal or irreducible $M_0$ \cite{dingtian92} and by Donaldson for general $M_0$. It can be
also formulated as the CM-weight introduced in
\cite{tian97}. If $\bG_0$ acts on $M_0$ trivially, we simply set $f_{M_0, L_0}(\bG_0)=0$.

In his thesis \cite{chili} (also see \cite{paultian06}), C. Li observed
\begin{lemm}
\lab{lemm:k-1}
For any algebraic subgroup $\bG_0=\{\sigma(t)\}_{t\in \CC^*} $ of $\bG$, we have
\begin{equation}
\lab{eq:k-2} \bF(\sigma(t))\,=\, - \,(f_{M_0, L_0}(\bG_0) \, -\, a(\bG_0)) \log |t|^2 \,+\,O(1)~~{\rm as}~~t\to 0,
\end{equation}
where $a(\bG_0)\in \QQ$ is non-negative and the equality holds if $M_0$
has no non-reduced components.
\end{lemm}
\begin{proof}
C. Li has pointed out that \eqref{eq:k-2} can be actually derived from \cite{tian97}. For the readers' convenience, we give a proof here by
using arguments from \cite{tian97}.

Define $\hat\cX$ as the set of all $(x,t)$ in $\CC P^N \times \CC$ satisfying: $x\in \sigma(t) (M)$ when $t\not=0$ and $x\in M_0$ when $t=0$.  It admits
a compactification $\cX$ as follows: There is an natural biholomorphism $\phi$ from $\cX_0=\hat\cX \backslash M_0$ onto $M\times \CC^* $ by $\phi(x,t)=(\sigma^{-1}(t)(x), t)$. Consider $\CC P^1$ as $\CC $ plus the point $\infty$, then we define
$$\cX\,=\,\hat\cX\cup _{\phi: \cX_0\simeq M\times \CC^* }(M\times \CC P^1\backslash \{0\}).$$
Clearly, $\cX$ admits a fibration over $\CC P^1$. Also, $L$ induces a relatively ample bundle $\cL$ over $\cX$: $\cL|_{\hat\cX} = \pi_1^*\cO_{\CC P^N}(1)|_{\hat\cX}$
and $\cL|_{M\times \CC P^1\backslash \{0\}} = \pi_1^*L^\ell$, where $\pi_i$ denotes the projection onto the $i$th factor.
Since $\sigma(s)\cdot\phi= \phi\cdot \sigma(s)$ for any $s\in \CC^*$,
we have a $G_0$-action on $\cX$:
$\sigma(s)(x,t)$ is equal to $(\sigma(s)(x), s\cdot t)$ on $\hat\cX$ and $ (x, s\cdot t)$ on $ M\times(\CC P^1\backslash \{0\})$.
Similarly, there is an natural lifting of $\bG_0$-action on $\cL$ which acts on $\cL|_{\hat\cX}$ as given and on 
$\cL|_{M\times \CC P^1\backslash \{0\}} = \pi_1^*L^\ell$ by $\sigma(s)(v,t)=(v,s\cdot t)$, where $v\in L$.

Let $p: \tilde \cX\mapsto \cX$ be a $\bG_0$-equivariant resolution and $\tilde\cL = p^*\cL$ (cf. \cite{kollar}). There is an induced fibration
$\pi:\tilde\cX\mapsto \CC P^1$, we denote $M_z=\pi^{-1} (z)$.
Choose a smooth Hermitian norm $h$ on $\tilde \cL$ over $\tilde\cX$ satisfying:

\vskip 0.1in
\noindent
(1) $h = p^*\pi_1^* h_1$ over $\pi^{-1}(\{ z\in \CC \,|\,|z| \le 1\})$, where $h_1$ denotes a fixed Hermitian metric on $\cO_{\CC P^N} (1)$
whose curvature is the Fubini-Study metric $\omega_{FS}$;

\vskip 0.1in
\noindent
(2) For any $z\in \CC^*$, the curvature form $R(h)$ of $h$ restricts to a
K\"ahler metric $\omega_z$ on $M_z$ satisfying:
$$\omega_z\,=\,\sigma(z)^*\omega_{FS}|_{M_z}~~{\rm for}~~|z|\le 1~~{\rm and}~~\omega_z\,=\,\ell \,\omega_0~~{\rm for}~~|z|\ge 2.$$
Here we regard $\tilde\cX\backslash p^{-1}(M_0)=\cX\backslash M_0$ as $M\times \CC P^1\backslash \{0\}$.
We may further assume that for $|z|\ge 2$, $h|_{M_z}$ is equal to a fixed norm $h_0$ on $L^\ell$ and for $|z|\le 1$,
$\sigma(z)^*h|_{M_z}= e^{-\ell\,\varphi_z}\, h_0$, where $\varphi_z$ is a corresponding K\"ahler potential, i.e.,
$$\frac{1}{\ell}\,\omega_z\, =\,\omega_0\,+\,\sqrt{-1}\,\partial\bar\partial \,\varphi_z.$$

Let $\cK=K_{\tilde\cX}\otimes \pi^*K_{\CC P^1}^{-1} $ be the relative canonical bundle of $\pi:\tilde\cX\mapsto \CC P^1$.
It has an induced Hermitian norm $k$ on $\cK^{-1}$ over $\CC P^1\backslash \{0\}$: For each $z\in \CC P^1\backslash \{0\}$,
$k |_{M_z} $ is given by the determinant of $\omega_z$. Clearly, the curvature form $R(k)$ of $k$ restricts to
${\rm Ric}(\omega_z)$ on each $M_z$.

Put $F(z) = \bM_{\omega_0}(\varphi_z)$, then $F$ is a continuous function on $\CC P^1\backslash \{0\}$, constant for $|z|\ge 2 $ and coincides with
$\bF$ for $|z|\le 1$.
%Clearly, $F(z)\,=\,\bF(\sigma(z))$ for $|z|\le 1$.
Following those direct computations exactly as we did for (8.5)
in \cite{tian97}, we can show that for any smooth function $\phi(z)$ with
support contained in $\CC P^1\backslash\{0\}$,
\begin{equation}\label{eq:k-3}
-\int_{\CC P^1} \,F \,\partial\bar\partial \phi\,=\,\frac{1}{V}\,\int_{\tilde\cX} \,\phi\,\left(R( k) \,-\, \frac{n\,\mu}{(n+1)\,\ell}\,R(h)\right)\wedge \left(\frac{1}{\ell}\,R(h)\right)^n,
\end{equation}
where $V\,=\,c_1(L)^n$. 

Let $\tilde\omega$ be a K\"ahler metric on $\tilde\cX$. We can construct another Hermitian metric $\tilde k $ on $\cK$ as we did for $k$. Then
the ratio $k/\tilde k $ is an non-negative function bounded from above. It follows from \eqref{eq:k-3}
\begin{equation}\label{eq:k-4}
-\int_{\CC P^1} \,(F -  \xi) \,\partial\bar\partial \phi\,=\,\frac{1}{V}\,\int_{\tilde\cX} \,\phi\,\left(R(\tilde k) \,-\, \frac{n\,\mu}{(n+1)\,\ell}\,R(h)\right)\wedge \left(\frac{1}{\ell}\,R(h)\right)^n,
\end{equation}
where
\begin{equation}\label{eq:k-5}
\xi (z)\,=\, \frac{1}{V}\,\int_{M_z}\, \left(\frac{k}{\tilde k} \right )\,\log \left (\frac{k}{\tilde k} \right )\,\tilde\omega^n.
\end{equation}
This is a bounded function, in fact, it is continuous in $z$.

Denote by $g_{\tilde\cX}$ and $g_B$ the Hermitian norms on $K_{\tilde\cX}^{-1}$ and $K_{\CC P^1}^{-1}$ induced by the metric $\tilde\omega$ on
$\tilde\cX$ and the Fubini-Study metric $\omega_{FS}$ on $\CC P^1$. Define
\begin{equation}
\label{eq:k-6}
\zeta (z)\,=\, \frac{1}{V}\,\int_{M_z}\,\log \left( \frac{\tilde k \cdot\pi^*g_B}{g_{\tilde\cX}}\right) \,\left(\frac{1}{\ell}\,R(h)\right)^n.
\end{equation}
It is easy to show (cf. \cite{tian97}\footnote{This function is denoted as $\psi_Z$ there.})
that $\zeta$ is bounded from above and extends across $0$ continuously if $M_0$ does not have components of multiplicity greater than $1$.
In fact, one can show (see \cite{paultian04} and \cite{paul08}) that
\begin{equation}\label{eq:k-7}
\zeta(t)\,=\,a(\bG_0)\,\log |t|^2 \,+\,O(1),~~~{\rm as}~t\to 0,
\end{equation}
where $a(\bG_0)\ge 0$ and $O(1)$ denotes a bounded quantity. It follows from \eqref{eq:k-4}
$$\int_{\CC P^1} (F -  \xi - \zeta)\, \partial\bar\partial \phi=\frac{1}{V}\,\int_{\tilde\cX} \phi\, \left(R(g_B)- R(g_{\tilde\cX}) + \frac{n\,\mu}{(n+1)\,\ell}\,
R(h)\right) \wedge \left(\frac{1}{\ell}\, R(h)\right)^n.$$
Let $\bL$ be the determinant line bundle $\det(\cE,\pi)$, where $\cE$ is defined by
$$2^{n+1}\,\ell^n\, V\,\cE\,=\, (\cK^{-1}\,-\,\cK)\otimes (\tilde \cL \,-\,\tilde \cL^{-1})^n\,-\, \frac{n \mu }{(n+1)\,\ell}\, (\tilde \cL \,-\,\tilde \cL^{-1})^{n+1}.$$
This line bundle $\bL$ was introduced in \cite{tian97} and called the CM-line bundle or polarization.\footnote{A different formulation of the CM-line bundle
was given in \cite{paultian06} which suits better for more general fibrations.}
By the Grothendick-Riemann-Roch Theorem, the first Chern class of $\bL$ is given by
$$c_1(\bL)\,=\,\ell ^{-n}\,\pi_* [c_1(\cK)\, c_1(\tilde \cL)^n\,+\, \frac{ n\mu}{(n+1)\,\ell}\,c_1(\tilde \cL)^{n+1}].$$
The corresponding degree is simply the CM-weight dated back to \cite{tian97}. 
It was proved in \cite{lixu11} that this CM-weight coincides with the generalized Futaki invariant $f_{M_0,L_0}(\bG_0)$.
Furthermore, there is a H\"older continuous norm $||\cdot||_B$ on $\bL$ over $\CC P^1$ whose curvature form is given by the push-forward
form
$$ \pi_*\left[\left(R(g_B) -  R(g_{\tilde\cX}) + \frac{n\,\mu}{(n+1)\,\ell}\,R(h)\right) \wedge \left (\frac{1}{\ell}\, R(h)\right )^n\right ].$$
Fix a unit $z_0\in \CC$ and $1\in \bL|_{z_0}$, then we set
$$S(z) \,=\, \sigma(t) (1)\,\in\,\bL|_z,~~~{\rm where}~z\,=\,\sigma(t)(z_0).$$
This defines a holomorphic section $S$ of $\bL$ over $\CC^*$ which extends to $\CC P^1\backslash\{0\}$, moreover, it is
non-zero at $\infty$. It follows from the above discussions
$$\partial \bar\partial (F -  \xi - \zeta + \log ||S||_B^2)\,=\,0~~{\rm on}~~\CC^*.$$
Since $F -  \xi - \zeta + \log ||S||_B$ is bounded near $\infty$, we conclude
\begin{equation}\label{eq:k-8}
F\,=\,\xi \,+\,\zeta\,-\,\log ||S||_B^2\,+\,c,
\end{equation}
where $c$ is a constant.

On the other hand, we can extend $S$ to be a meromorphic section of $\bL$ with an zero or pole of order $\pm\,f_{M_0,L_0}(\bG_0)$ at $0$.
Then \eqref{eq:k-2} follows from \eqref{eq:k-8}, \eqref{eq:k-7} and the facts that $\xi$ is bounded and $F(t)=\bF(\sigma(t))$ for $|t|\le 1$.

\end{proof}

We can use the arguments in the proof to identify $f_{M_0, L_0}(\bG_0) \, -\, a(\bG_0)$ with a generalized Futaki invariant
of some special degeneration of $M$. Let us first describe such a degeneration. We will adopt the notations in last proof. Note that
there is an natural fibration $\hat\pi: \hat \cX\mapsto \CC$.
%For any positive integer $m$, define
%$$\hat\cX_m\,=\,\{\,(y,t)\,|\,y\in \hat\cX,~~\hat\pi(y)= t^m\,\}.$$
%Clearly, $\hat\cX_1=\hat\cX$ and there is an natural fibration $\hat\pi:\hat\cX_m\mapsto \CC$.
It was shown in \cite{lixu11} (also see \cite{arezzolanavevedova})
that there is a $\bG_0$-equivariant semi-stable reduction $\pi':\cX'\mapsto \CC$ of $\hat\cX$ whose generic fiber is biholomorphic to $M$.\footnote{
In fact, we do not need $\cX'$ to be a semi-stable reduction in the subsequent discussions. It is sufficient if the central fiber of
$\cX'$ is free of multiple components. Then such a $\cX'$ can be taken to be the normalization of a base change of $\hat\cX$.}
This implies that
the central fiber $M_0'=\pi'^{-1}(0)$ is a singular variety with normal crossings.
Furthermore, there is an natural map $q:\cX'\mapsto\hat\cX$ of degree $m$ with $q(M_0')=M_0$.
Then we have a generalized Futaki invariant $f_{M_0',L_0'}(\bG_0)$ associated to the degeneration $\pi':\cX'\mapsto \CC$.

\begin{lemm}
\lab{lemm:k-2}
For any $\bG_0$ above, we have
\begin{equation}
\lab{eq:k-9} \bF(\sigma(t))\,=\, - \,\frac{1}{m}\,f_{M_0', L_0'}(\bG_0) \, \log |t|^2 \,+\,O(1)~~{\rm as}~~t\to 0.
\end{equation}
In particular, we have 
\begin{equation}\label{eq:k-9'}
f_{M_0', L_0'}(\bG_0)\,=\,m\,(f_{M_0, L_0}(\bG_0) \, -\, a(\bG_0)).
\end{equation}
\end{lemm}
\begin{proof}
We will use the arguments in the proof of last lemma to prove \eqref{eq:k-9} .

Let $\cX$ be the compactification of $\hat\cX$ and $\cL$ be the line bundle over $\cX$
constructed in the proof of Lemma \ref{lemm:k-1}. Then it admits a $\bG_0$-equivariant semi-stable reduction $\pi': \cX'_s\mapsto \CC P^1$
such that it is a compactification of $\cX'$ with smooth fiber over $\infty\in \CC P^1$ and admits a holomorphic map $q: \cX'_s\mapsto \cX$ of degree $m$.
To prove \eqref{eq:k-9} which is equivalent to \eqref{eq:k-9'}, we simply argue as we did in the proof of Lemma \ref{lemm:k-1} with $p:\tilde\cX\mapsto \cX$ 
replaced by $q:\cX'_s\mapsto \cX$.

The norm $h$ in the proof of last lemma induces a Hermitian norm, still denoted by $h$, on $\cL'=q^*\cL$ over $\cX'_s$.
As before, by identifying $\cX'_s\backslash M_0'$ with $M\times\CC P^1\backslash\{0\}$, the curvature $R(h)$ restricts to a
K\"ahler metric $\omega_z$ on $M_z'$, which is simply $\pi'^{-1} (z)$, for any $z\in \CC^*$ satisfying:
$$\omega_z\,=\,\sigma(z)^*\omega_{FS}|_{M_z'}~{\rm for}~|z|\le 1~~{\rm and}~~\omega_z\,=\,\ell \,\omega_0~{\rm for}~|z|\ge 2.$$
Further, we have $h|_{M_z'}=h_0$ on $L^\ell$ for $|z|\ge 2$ and $\sigma(z)^*h|_{M_z'}= e^{-\ell\,\varphi_z}\, h_0$ for $|z|\le 1$.

Let $\cK=K_{\cX'_s}\otimes \pi'^*K_{\CC P^1}^{-1} $ be the relative canonical bundle of $\pi':\cX'_s\mapsto \CC P^1$.
It has an induced Hermitian norm $k$ on $\cK^{-1}$ over $\CC P^1\backslash \{0\}$: For each $z\in \CC P^1\backslash \{0\}$,
$k |_{M_z'} $ is given by the determinant of $\omega_z$. Clearly, the curvature form $R(k)$ of $k$ restricts to
${\rm Ric}(\omega_z)$ on each $M_z'$.

It follows from \eqref{eq:k-3}
\begin{equation}\label{eq:k-3'}
-m \,\int_{\CC P^1} \,F \,\partial\bar\partial \phi\,=\,\frac{1}{V}\,\int_{\cX_s'} \,\phi\,\left(R( k) \,-\, \frac{n\,\mu}{(n+1)\,\ell}\,R(h)\right)\wedge \left(\frac{1}{\ell}\,R(h)\right)^n.
\end{equation}

Let $\omega'$ be a K\"ahler metric on $\cX_s'$. We can construct another Hermitian metric $k'$ on $\cK$ as we did for $k$. Then
the ratio $k/k '$ is an non-negative function bounded from above. It follows from \eqref{eq:k-3}
\begin{equation}\label{eq:k-10}
-\int_{\CC P^1} \,(m\,F -  \xi) \,\partial\bar\partial \phi\,=\,\frac{1}{V}\,\int_{\cX_s'} \,\phi\,\left(R(k') \,-\, \frac{n\,\mu}{(n+1)\,\ell}\,R(h)\right)\wedge \left(\frac{1}{\ell}\,R(h)\right)^n,
\end{equation}
where
\begin{equation}\label{eq:k-11}
\xi (z)\,=\, \frac{1}{V}\,\int_{M_z'}\, \left(\frac{k}{ k'} \right )\,\log \left (\frac{k}{ k'} \right )\,\omega'^n.
\end{equation}
This is again a bounded function.

Denote by $g'$ and $g_B$ the Hermitian norms on $K_{\cX'_s}^{-1}$ and $K_{\CC P^1}^{-1}$ induced by the metric $\omega'$ on
$\cX_s'$ and the Fubini-Study metric $\omega_{FS}$ on $\CC P^1$. Define
\begin{equation}
\label{eq:k-12}
\zeta (z)\,=\, \frac{1}{V}\,\int_{M_z'}\,\log \left( \frac{ k' \cdot\pi^*g_B}{g'}\right) \,\left(\frac{1}{\ell}\,R(h)\right)^n.
\end{equation}
Since $M_0'$ has no multiple components, $\zeta$ is bounded. It follows from \eqref{eq:k-10}
$$\int_{\CC P^1} (m\,F -  \xi - \zeta) \partial\bar\partial \phi=\frac{1}{V}\,\int_{\cX_s'} \phi \left(R(g_B)- R(g_{\cX_s'}) + \frac{n \mu}{(n+1) \ell}\,
R(h)\right) \wedge \left(\frac{1}{\ell} R(h)\right)^n.$$
Let $\bL$ be the determinant line bundle $\det(\cE',\pi)$, where $\cE'$ is defined by
$$2^{n+1}\,\ell^n\, V\,\cE'\,=\, (\cK^{-1}\,-\,\cK)\otimes (\cL' \,-\, \cL'^{-1})^n\,-\, \frac{n \mu }{(n+1)\,\ell}\, ( \cL '\,-\,\cL'^{-1})^{n+1}.$$

By the Grothendick-Riemann-Roch Theorem, the first Chern class of $\bL$ is given by
$$c_1(\bL)\,=\,\ell ^{-n}\,\pi'_* [c_1(\cK)\, c_1(\cL')^n\,+\, \frac{n \mu}{(n+1)\,\ell}\,c_1(\cL')^{n+1}].$$
The corresponding degree is the CM-weight which is equal to the invariant $f_{M_0',L_0'}(\bG_0)$.
Furthermore, there is a H\"older continuous norm $||\cdot||_B$ on $\bL$ over $\CC P^1$ whose curvature form is given by the push-forward
form
$$ \pi'_*\left[\left( R(g_B) - R(g_{\cX_s'}) + \frac{n\,\mu}{(n+1)\,\ell}\,R(h)\right) \wedge \left (\frac{1}{\ell}\, R(h)\right )^n\right ].$$
Fix a unit $z_0\in \CC$ and $1\in \bL|_{z_0}$, then we set
$$S(z) \,=\, \sigma(t) (1)\,\in\,\bL|_z,~~~{\rm where}~z\,=\,\sigma(t)(z_0).$$
This defines a holomorphic section $S$ of $\bL$ over $\CC^*$ which extends to $\CC P^1\backslash\{0\}$, moreover, it is
non-zero at $\infty$. It follows from the above discussions
$$\partial \bar\partial (m\,F -  \xi - \zeta + \log ||S||_B^2)\,=\,0~~{\rm on}~~\CC^*.$$
Since $m\,F -  \xi - \zeta + \log ||S||_B$ is bounded near $\infty$, we conclude
\begin{equation}\label{eq:k-12'}
m\,F\,=\,\xi \,+\,\zeta\,-\,\log ||S||_B^2\,+\,c,
\end{equation}
where $c$ is a constant.

On the other hand, we can extend $S$ to be a meromorphic section of $\bL$ with an zero or pole of order $\pm\,f_{M_0',L_0'}(\bG_0)$ at $0$. Then \eqref{eq:k-9} follows from \eqref{eq:k-12'} and the facts that both $\xi$ and $\zeta$ are bounded.

\end{proof}

\begin{rema}
One can also prove \eqref{eq:k-9} by using the equivariant Riemann-Roch Theorem. 
\footnote{I learned from Chenyang Xu that \eqref{eq:k-9'}, and consequently, \eqref{eq:k-9}, 
can be also proved by a purely algebraic method.}

%\footnote{If $M_0$ is smooth, it is exactly the same as
%Donaldson's identifying his
%reformulation of the Futaki invariant with the original one (see\cite{donaldson02}).}
\end{rema}

It follows from Lemma \ref{lemm:k-2}

\begin{theo}\label{th:k-2}
If $(M, L)$ is K-stable, then $\bF$ is proper along any one-parameter algebraic subgroup $\bG_0$ of $\bG$ unless $\bG_0$ preserves $M$, i.e.,
it is contained in the automorphism group of $M$.
\end{theo}

\section{Proving Theorem \ref{th:main} when ${\rm Aut_0}(M,L)=\{1\}$}

In view of Theorem \ref{th:k-2}, the K-stability implies that $\bF$ is proper along any one-parameter algebraic subgroup of $\bG$. Hence,
our problem is whether or not the properness of $\bF$ on $\bG$ follows from the properness of $\bF$ along any one-parameter algebraic subgroup of $\bG$.
This is an algebraic problem in nature.

As in classical Geometric Invariant Theory, we deduce the CM-stability from the K-stability in two steps. For simplicity, we assume
${\rm Aut_0}(M,L)=\{1\}$ in this section and explain how to adapt the proof to the general case in the next section.

\begin{lemm}\label{lemm:k-3}
Let $\bT$ be any maximal algebraic torus of $\bG$. If
the restriction $\bF|_{\bT}$ is proper in the sense of \eqref{eq:cm-defi}, then $M$ is CM-stable with respect to $L^\ell$.
\end{lemm}
\begin{proof}
We prove it by contradiction. Suppose that we have a sequence $\sigma_i\in\bG$ such that
$\bF(\sigma_i)$ stay bounded while $\bJ(\sigma_i)$ diverge to $\infty$.

Recall the Cartan decomposition: $\bG= \bK \cdot\bT\cdot\bK$, where $\bK = U(N+1)$.
Write $\sigma_i = k_i' t_i k_i$ for $k_i,k_i'\in \bK$ and $t_i\in \bT$. Then
we have that $\bF(t_i k_i)= \bF(\sigma_i)$ stay bounded while $\bJ(t_ik_i)=\bJ(\sigma_i)$ diverge to
$\infty$.

On the other hand, since each $k_i$ is represented by unitary matrix, we can show easily
\begin{equation}\label{eq:k-13}
|\psi_{t_i} - \psi_{t_i k_i}|\,\le\,\log (N+1).
\end{equation}
Write $\bD(t_i)=\bF(\psi_{t_i}) \,-\, \bF(\psi_{t_ik_i})$. Using \eqref{eq:k-13} and the definition of the K-energy, we can deduce
$$
\bD(t_i)\,=\,\frac{1}{\ell^n\,V}\,\left(\int_M \,\log \left( \frac{\omega_{t_i}^n}{\omega_0^n}\right )\,\omega_{t_i}^n -
\int_M \,\log \left( \frac{\omega_{t_ik_i}^n}{\omega_0^n}\right )\,\omega_{t_ik_i}^n\right) +\cO(1).
$$
The integrals on the right side are equal to
$$\int_M \,\log \left( \frac{\omega_{t_i}^n}{\omega_{t_ik_i}^n}\right )\,\omega_{t_i}^n +
\int_M \,\log \left( \frac{\omega_{t_ik_i}^n}{\omega_0^n}\right )\,(\omega_{t_i}^n-\omega_{t_ik_i}^n).
$$
The first integral above is bounded from below while the second is equal to
\begin{equation}\label{eq:k-14}
\ell\,\int_M\, (\psi_{t_i}-\psi_{t_ik_i})\,\left({\rm Ric}(\omega_0)-{\rm Ric}(\omega_{t_ik_i})\right)\,\wedge \sum_{a=1}^{n-1} \omega_{t_i}^a\wedge \omega_{t_ik_i}^{n-a}.
\end{equation}
Using \eqref{eq:k-13} and the fact that ${\rm Ric}(\omega_{t_ik_i})$ is bounded from above, we can show that the integral in \eqref{eq:k-14}
is uniformly bounded. It follows that $\bD(t_i)$ is bounded from below. Similarly, using the fact that ${\rm Ric}(\omega_{t_i})$ is bounded from above,
one can show that $\bD(t_i)$ is bounded from above. Therefore, we have
$$|\bF(t_i)\,-\,\bF(t_ik_i)|\,\le\, C.$$
It follows that $\bF(t_i)$ stay bounded while $\bJ(t_i)$ diverge to
$\infty$. We get a contradiction.
\end{proof}

Therefore, in order to prove Theorem \ref{th:main}, we only need to prove that $\bF$ is proper on the maximal algebraic torus $\bT$.
The remaining arguments are identical to corresponding parts in \cite{tian12} or
\cite{tian13} which is based on S. Paul's works \cite{paul12} and \cite{paul13}.

First we recall the Chow coordinate and Hyperdiscriminant of $M$ (\cite{paul08}):
Let $G(N-n-1,N)$ the Grassmannian of all $(N-n-1)$-dimensional subspaces in
$\CC P^N$. We define
\begin{equation}
\label{eq:chow-1}
Z_M\,=\,\{\, P \in G(N-n-1,N)\,|\,P\cap M\,\not=\,\emptyset\,\}.
\end{equation}
Then $Z_M$ is an irreducible divisor of $G(N-n-1,N)$ and determines a non-zero homogeneous polynomial $R_M \in \CC[M_{(n+1)\times (N+1)}] $,
unique modulo scaling, of degree $(n+1) d$, where $M_{k\times l}$ denotes the space of all $k\times l$ matrices.
We call $R_M$ the Chow coordinate or the $M$-resultant of $M$.

Next consider the Segre embedding:
$$M\times \CC P^{n-1} \subset \CC P^N\times \CC P^{n-1} \mapsto \mathbb{P}(M_{n\times (N+1)}^\vee),$$
where $M_{k\times l}^\vee$ denotes its dual space of $M_{k\times l}$. Then we define
\begin{equation}
\label{eq:chow-2}
Y_M\,=\,\{\, H\, \subset\,\mathbb{P}(M_{n\times (N+1)}^\vee)\,|\, T_p(M\times \CC P^{n-1}) \,\subset\, H~{\rm for~some}~p\,\}.
\end{equation}
Then $Y_M$ is a divisor in $\mathbb{P}(M_{n\times (N+1)}^\vee)$ of degree $\bar d\,=\,(n (n+1) - \mu )\, d$. This determines
a homogeneous polynomial
$\Delta_M$ in $\CC[M_{n\times (N+1)}]$, unique modulo scaling, of degree $\bar d$. We call $\Delta_M$ the hyperdiscriminant of $M$.

Set
$$r = (n+1) \,d \bar d, ~~\bV\,=\,C_r [M_{(n+1)\times (N+1)}],~~\bW\,=\,C_r [M_{n \times (N+1)}],$$
where $C_r[\CC^k]$ denotes the space of homogeneous polynomials of degree $r$ on $\CC^k$.
Following \cite{paul12}, we associate $M$ with the pair $(R(M), \Delta(M))$ in $\bV\times \bW$, where
$R(M)=R_M^{\bar d}$ and $\Delta(M)=\Delta_M^{(n+1)d}$.

Fix norms on $\bV$ and $\bW$, noth denoted by $||\cdot||$ for simplicity, we set
\begin{equation}
\label{eq:cm-8}
p_{v,w}\,=\,\log||w||\,-\,\log ||v||.
\end{equation}

The following was first observed by S. Paul, but the proof below was presented by myself in \cite{tian13}.
\begin{lemm}
\label{lemm:k-4}
Let $(\sigma, B)\,\mapsto\, \sigma( B): \bG\times \gl\,\mapsto\, \gl$ be the natural representation by left multiplication, where
$\gl$ denotes the space of all $(N+1)\times(N+1)$ matrices. Then we have
\begin{equation}\label{eq:cm-7}
|\,\bJ(\sigma) \,-\, p_{R(M), I^r} (\sigma)\,| \,\le\, C,
\end{equation}
where $I$ is the identity in $\gl$ and $I^r\in \bU=\gl^{\otimes r}$.
\end{lemm}
\begin{proof}
It is known (cf. \cite{paul04})
$$(n+1)\,\bJ (\sigma)\,= \,  (n+1)\, \int_M \,\psi_\sigma \,\omega_0^n\,-\, \log ||\sigma (R_M)||^2.
$$
This is equivalent to
\begin{equation}\label{eq:J-1}
(n+1) \,\bar d  \,\bJ (\sigma)\,= \,  r \, \int_M \,\psi_\sigma \,\frac{\omega_0^n}{d}\,-\, \log ||\sigma (R(M))||^2.
\end{equation}
If we write $\sigma\in \SL(N+1,\CC)$ as a $(N+1)\times (N+1)$-matrix $(\vartheta_{ij})$ with determinant one, then
the Hilbert-Schmidt norm of $\sigma$ is given by
$$||\sigma||^2\,=\,\sum_{i,j=0}^{N} |\vartheta|^2.$$
Clearly, we have
$$\psi_\sigma\,=\, \log  \left (\sum_{i=0}^{N} \,||\sum_{j=0}^N \vartheta_{ij} S_j ||^2\right ),$$
where $\{S_j\}_{0\le j\le N}$ is an orthonormal basis. By direct computations, we can easily show
$$\left |\,\log ||\sigma||^2\,-\,\int_M\, \log \left (\sum_{i=0}^{N}\, ||\sum_{j=0}^N \vartheta_{ij} S_j ||^2\right ) \,\frac{\omega_0^n}{d}\,\right|  \,\le\,C.$$
Combining the above two with \eqref{eq:J-1}, we get \eqref{eq:cm-7}.
\end{proof}

\begin{lemm}\label{lemm:k-5}
Let $\bV$, $\bW$ and $\bU$ be as above and $\bG_0$ be an one-parameter algebraic subgroup.
Then $\bF$ is not proper on $\bT$ (resp. $\bG_0$) if and only if
the orbit of $[R(M),\Delta(M)]\times [R(M),I^r]$ under $\bT$
(resp. $\bG_0$) has a limit point in
$$\left ( \mathbb{P} (\bV\oplus \bW)\backslash \mathbb{P} (\{0\}\oplus \bW)\right )\times \mathbb{P}(\{0\}\times \bU).$$
\end{lemm}
\begin{proof} First we note that $\left (\mathbb{P}(\bV\oplus \bW)\backslash \mathbb{P}(\{0\}\oplus \bW)\right )\times \mathbb{P}(\{0\}\times \bU)$ is $\bT$-invariant.
It follows from \cite{paul08} that for all $\sigma\in \bG$, we have
\begin{equation}\label{eq:cm-8}
|\,\bF(\sigma ) \,-\, a_n\,p_{R(M),\Delta(M)}(\sigma) \,|\, \leq\,C,
\end{equation}
where $a_n > 0 $ and $C$ are uniform constants.

By Lemma \ref{lemm:k-4} and \eqref{eq:cm-8}, we see that if $\bF$ is not proper on $\bT$ (resp. $\bG_0$), then there are $\sigma_i\in\bT$ (resp. $\bG_0$) such that
$p_{R(M), \Delta (M)}(\sigma_i)$ stay bounded while $p_{R(M), I^r} (\sigma_i)$ goes to $\infty$.
In \cite{paul08}, S. Paul showed
$$p_{R(M),\Delta(M)}(\sigma)\,=\,\log \tan^2 d(\sigma ([R(M),\Delta(M)]), \sigma([R(M), 0]))$$
and
$$p_{R(M),I^r}(\sigma)\,=\,\log \tan^2 d(\sigma ([R(M),I^r]), \sigma ([R(M),0])),
$$
where $d(\cdot,\cdot)$ denotes the distance in $\mathbb{P}(\bV\oplus\bW)$ with respect to the Fubini-Study metric.
Therefore, the limits of $\sigma_i([R(M), I^r])$ lie in $\mathbb{P}(\{0\}\oplus \bU)$ while limits of $\sigma_i([R(M), \Delta(M)])$ stay in
$\mathbb{P}(\bV\oplus \bW)\backslash \mathbb{P}(\{0\}\oplus \bW)$.

The other direction can be easily proved by reversing the above arguments.
The lemma is proved.
\end{proof}
Now we deduce Theorem \ref{th:main} from Lemma \ref{lemm:k-5}. If $M$ is not CM-stable with respect to $L^\ell$,
then there are $v\in \bV,w\in\bW,u\in \bU$ with $u\not= 0, v\not= 0$ satisfying:
If we denote $ y=[v,w]\times[0,u]$
and $x=[R(M),\Delta(M)]\times [R(M),I^r]$,
then $y$ is in the closure of the $\bT$-orbit of $x$.

Choose $\bT$-invariant hyperplanes
$\bV_0\subset\bV$ and $\bU_0\subset \bU$
%, which can be naturally identified with $\mathbb{P}(\bV)\backslash \mathbb{P}(\bV_0)$ 
%and $\mathbb{P}(\bU)\backslash \mathbb{P}(\bU_0)$, 
such that $x\in \bE$ and $y\in \bE_0$, where
$$E = \left (\mathbb{P}(\bV\oplus\bW)\backslash \mathbb{P}(\bV_0\oplus\bW)\right)
\times\left(\mathbb{P}(\bV\oplus\bU)\backslash \mathbb{P}(\bV\oplus\bU_0)\right )$$
and
$$\bE_0= \left (\mathbb{P}(\bV\oplus\bW)\backslash \mathbb{P}(\bV_0\oplus\bW)\right) \times \left (\mathbb{P}(\{0\}\oplus\bU)\backslash \mathbb{P}(\{0\}\oplus\bU_0)\right).$$
Clearly, $\bE_0$ is a closed subspace of $\bE$ and the orbit $\bT  y$ lies in $\bE_0$. Also both $\bE$ and $\bE_0$ are affine. They
are actually isomorphic to $\bV_0\times\bW\times \bV\times\bU_0$ and $\bV_0\times\bW\times \{0\}\times\bU_0$, respectively.

By taking an orbit in the closure of $\bT y$ if necessary, we may assume that $\bT y$ is closed in $\bE_0$. Then, by a well-known
result of Richardson (cf. \cite{paul12} and also \cite{tian13}), there is an one-parameter algebraic subgroup $\bG_0$ such that the closure of $\bG_0 x$ contains a
point in $\bE_0$ which is a subset of 
$$ \left (\mathbb{P}(\bV\oplus \bW)\backslash \mathbb{P}(\{0\}\oplus \bW)\right )\times \mathbb{P}(\{0\}\times \bU) .$$ 
By Lemma \ref{lemm:k-5} and Theorem \ref{th:k-2}, this contradicts to the K-stability of $M$.
Thus, the proof of Theorem \ref{th:main} is completed.

\section{Proving Theorem \ref{th:main} in general cases}

In this section, we prove Theorem \ref{th:main} in full generality. It is clear that only Lemma \ref{lemm:k-3} and Lemma \ref{lemm:k-5} need to be modified, all other arguments in last section go through without change.

First we prove a generalized version of Lemma \ref{lemm:k-3}.
\begin{lemm}\label{lemm:k-6}
Let $\bT$ be any maximal algebraic torus of $\bG$. If
the restriction $\bF|_{\bT}$ is proper in the following sense: For any sequence $\sigma_i\in \bT$,
\begin{equation}\label{eq:cm-4-1}
\bF(\sigma_i) \rightarrow\infty~{\rm whenever}~\inf_{\tau \in Aut(M,L)}\bJ(\sigma_i\tau) \rightarrow \infty,
\end{equation}
then $M$ is CM-stable with respect to $L^\ell$.
\end{lemm}
\begin{proof}
We also prove it by contradiction. Suppose that we have a sequence $\sigma_i\in\bG$ such that
$\bF(\sigma_i)$ stay bounded while $\inf_{\tau \in Aut(M,L)}\bJ(\sigma_i\tau)$ diverge to $\infty$.

As before, we use the Cartan decomposition: $\bG= \bK \cdot\bT\cdot\bK$, where $\bK = U(N+1)$, and
write $\sigma_i = k_i' t_i k_i$ for $k_i,k_i'\in \bK$ and $t_i\in \bT$. Then
we have that $\bF(t_i k_i)= \bF(\sigma_i) $ stay bounded while $\bJ(t_ik_i\tau)=\bJ(\sigma_i\tau)$ diverge to
$\infty$.

Using the arguments in the proof of Lemma \ref{lemm:k-3}, we can also prove that for some constant $C > 0$,
$$|\bF(t_i)\,-\,\bF(t_ik_i)|\,\le\, C.$$
It remains to prove
$$\inf_{\tau \in Aut(M,L)}\bJ(t_i\tau )\,\rightarrow\,\infty.$$
Each $k_i$ is represented by a unitary matrix $(\gamma_{a b})$ with $|\gamma_{a b}| \le 1$. Let $S_a = t_i(z_a)$, where  $a=0,\cdots, N$, and
$z_a$ is the $a$-th coordinate of $\CC P^N$, then all these $S_a$ form a basis of $H^0(M, L^\ell)$ and we have
$$\psi_{t_i}\,=\,\log (\sum_{a=0}^N\, ||S_a||^2_0),~~~\psi_{t_i k_i}\,=\,\log (\sum_{a=0}^N\, ||\sum_{b=0}^N \gamma_{ab} S_b||^2_0).
$$
It follows
\begin{equation}\label{eq:cm-4-2}
|\psi_{t_i} - \psi_{t_i k_i}|\,\le\,\log (N+1).
\end{equation}
Since $\tau$ is an automorphism of $M$, we have
\begin{equation}\label{eq:cm-4-3}
\psi_{t_i\tau}\,=\,\varphi_\tau + \psi_{t_i}\cdot\tau,~~~\psi_{t_i k_i\tau}\,=\,\varphi_\tau + \psi_{t_ik_i}\cdot\tau,
\end{equation}
where $\varphi_\tau$ is a function satisfying:
$$\tau^*\omega_0= \omega_0 +\sqrt{-1}\partial\bar\partial \,\varphi_\tau.$$

It follows from \eqref{eq:cm-4-2} and \eqref{eq:cm-4-3}
$$
|\psi_{t_i\tau} - \psi_{t_i k_i\tau }|\,\le\,\log (N+1).
$$
This implies that $|\bJ(t_i\tau)-\bJ(t_ik_i\tau)|$ is uniformly bounded.
Therefore, $\bF(t_i)$ stay bounded while $\inf_{\tau\in Aut(M,L)} \bJ(t_i\tau)$ diverge to
$\infty$. We get a contradiction.
\end{proof}

In the following, we will fix a maximal algebraic torus $\bT_0$ in $Aut_0(M,L)$.
We will prove that the K-stability implies
\begin{equation}\label{eq:cm-4-5}
\bF(t_i) \rightarrow\infty~{\rm whenever}~\inf_{\tau \in \bT_0}\bJ(t_i\tau) \rightarrow \infty, ~~~\forall\,\{t_i\}\subset \bT.
\end{equation}
Clearly, it follows from this and Lemma \ref{lemm:k-6} that $M$ is CM-stable with respect to $L^\ell$.

We will adopt the notations from last section. Choose an algebraic subtorus $\bT_1$ of $\bT$ such that $\bT =\bT_0\cdot \bT_1$ and $\bT_1$ is transversal $\bT_0$.

\begin{lemm}\label{lemm:k-7}
Let $\bV$, $\bW$ and $\bU$ be those in Lemma \ref{lemm:k-5}. If \eqref{eq:cm-4-5} is false,
then the orbit of $[R(M),\Delta(M)]\times [R(M),I^r]$ under $\bT_1$
has a limit point in
$$\left (\mathbb{P}(\bV\oplus \bW)\backslash \mathbb{P} (\{0\}\oplus \bW)\right )\times \mathbb{P}(\{0\}\times \bU).$$
\end{lemm}
\begin{proof} First we note that $\left (\mathbb{P}(\bV\oplus \bW)\backslash \mathbb{P} (\{0\}\oplus \bW)\right )\times \mathbb{P}(\{0\}\times \bU)$ is $\bT_1$-invariant.
If \eqref{eq:cm-4-5} is false, then
there is a sequence $\{t_i\}\subset \bT$ such that
$\bF(t_i)$ stays bounded while $\inf_{\tau \in \bT_0}\bJ(t_i\tau) $ diverge to $\infty$.
Write $t_i = s_i \tau_i $ for $s_i\in \bT_1$ and $\tau_i\in\bT_0$, then $\bF(s_i)=\bF(t_i)$ stays bounded while
$$\bJ(s_i)\,=\,\bJ(t_i\tau_i^{-1}) \,\ge\,\inf_{\tau \in \bT_0}\bJ(t_i\tau) \,\rightarrow\,\infty.$$
It follows from Lemma \ref{lemm:k-4} and \eqref{eq:cm-8} that $p_{R(M), \Delta (M)}(s_i)$ stay bounded while $p_{R(M), I^r} (s_i)$ goes to $\infty$.
Then, as we argued in the proof of Lemma \ref{lemm:k-5},
the limits of $s_i([R(M), I^r])$ lie in $\mathbb{P}(\{0\}\times \bU)$
while limits of $s_i([R(M), \Delta(M)])$ lie in $\mathbb{P}(\bV\oplus \bW)\backslash \mathbb{P} (\{0\}\oplus \bW)$.
The lemma is proved.
\end{proof}

Now we deduce Theorem \ref{th:main} in general cases from Lemma \ref{lemm:k-7}.

If $M$ is not CM-stable with respect to $L^\ell$,
there are $v\in \bV,w\in\bW,u\in \bU$ with $u\not=0, v\not= 0$ satisfying:  $y$ is in the closure of the $\bT_1$-orbit of
$x$, where $ y=[v,w]\times[0,u]$ and $x =[R(M),\Delta(M)]\times [R(M),I^r]$ as before.

Choose $\bT_1$-invariant hyperplanes
$\bV_0\subset\bV$ and $\bU_0\subset \bU$
%, which can be naturally identified with $\mathbb{P}(\bV)\backslash \mathbb{P}(\bV_0)$ 
%and $\mathbb{P}(\bU)\backslash \mathbb{P}(\bU_0)$, 
such that $x\in \bE$ and $y\in \bE_0$, where
$$E = \left (\mathbb{P}(\bV\oplus\bW)\backslash \mathbb{P}(\bV_0\oplus\bW)\right)
\times\left(\mathbb{P}(\bV\oplus\bU)\backslash \mathbb{P}(\bV\oplus\bU_0)\right )$$ 
and
$$\bE_0= \left (\mathbb{P}(\bV\oplus\bW)\backslash \mathbb{P}(\bV_0\oplus\bW)\right) \times \left (\mathbb{P}(\{0\}\oplus\bU)\backslash \mathbb{P}(\{0\}\oplus\bU_0)\right).$$
As before, $\bE_0$ is a closed subspace of $\bE$, the orbit $\bT_1  y$ lies in $\bE_0$ and both $\bE$ and $\bE_0$ are affine.

By taking an orbit in the closure of $\bT_1 y$ if necessary, we may assume that $\bT _1 y$ is closed in $\bE_0$. Then, by a well-known
result of Richardson (cf. \cite{paul12} and also \cite{tian13}), there is an one-parameter algebraic subgroup $\bG_0$ of $\bT_1$ such that the closure of $\bG_0 x$ contains a
point in $\bE_0$ which is a subset of $\left (\mathbb{P}(\bV\oplus \bW)\backslash \mathbb{P} (\{0\}\oplus \bW)\right )\times \mathbb{P}(\{0\}\times \bU)$.
By Lemma \ref{lemm:k-5} and Theorem \ref{th:k-2}, this contradicts to the K-stability of $M$. Thus, we have completed
the proof of Theorem \ref{th:main} in general cases.

%%%%%%%%%%%%%%
%%%%%%%%%%%%%%

\begin{thebibliography}{plain}
\frenchspacing
\bibitem[ALV09]{arezzolanavevedova}
Arezzo, C., Lanave, G. and Vedova. A.: Singularities and K-semistability. Int. Math. Res. Notices, 2012, no. 4, 849-869.

\bibitem[Do02]{donaldson02}
Donaldson, S:  Scalar curvature and stability of toric varieties. J. Diff. Geom., {\bf 62} (2002), 289-349.

\bibitem[DT92]{dingtian92}
Ding, W. and Tian, G.: K\"ahler-Einstein metrics and the
generalized Futaki invariants. Invent. Math., {\bf 110} (1992), 315-335.


\bibitem[Fu83]{futaki83}
Futaki, A.: An obstruction to the existence of Einstein-K\"ahler metrics. Inv. Math., {\bf 73} (1983), 437-443.

\bibitem[Ko07]{kollar} Koll\'ar, J.: Lectures on resolution of singularities. Annals of Mathematics Studies, {\bf 166}.
Princeton University Press, Princeton, NJ, 2007.

\bibitem[Li12]{chili}
Li, Chi: K\"ahler-Einstein metrics and K-stability. Princeton thesis, May, 2012.

\bibitem[LX11]{lixu11}
Li, Chi and Xu, Chengyang: Special test configurations and K-stability of Fano varieties. Preprint, arXiv:1111.5398.

\bibitem[Pa04]{paul04}
Paul, S.: Geometric analysis of Chow Mumford stability. Adv. Math. {\bf 182} (2004), no. 2, 333-356.

\bibitem[Pa08]{paul08}
Paul, S.: Hyperdiscriminant polytopes, Chow polytopes, and Mabuchi energy asymptotics. Ann. of Math. (2) {\bf 175} (2012), no. 1, 255-296.

\bibitem[Pa12]{paul12}
Paul, S.:  A Numerical Criterion for K-Energy maps of Algebraic Manifolds. Preprint, arXiv:1210.0924.

\bibitem[Pa13]{paul13}
Paul, S.: Stable Pairs and Coercive Estimates for The Mabuchi Functional. Preprint, arXiv:1308.4377.

\bibitem[Pe02]{perelman}
Perelman, G.: The entropy formula for the Ricci flow and its geometric applications. Preprint, arXiv:0211159.

\bibitem[PT04]{paultian04}
Paul, S. and Tian, G.: Analysis of geometric stability. Int. Math. Res. Notices., {\bf 48} (2004), 2555–2591.

\bibitem[PT06]{paultian06}
Paul, S. and Tian, G.: CM stability and the generalized Futaki invariant II. Astérisque No. {\bf 328} (2009), 339-354.

\bibitem[Ti97]{tian97}
Tian, G: K\"ahler-Einstein metrics with positive scalar
curvature. Invent. Math., {\bf 130} (1997), 1-39.

\bibitem[Ti98]{tian98}
Tian, G: Canonical Metrics on K\"ahler Manifolds.
Lectures in Mathematics ETH Z\"urich, Birkh\"auser Verlag, 2000.

\bibitem[Ti09]{tian09}
Tian, G.: Einstein metrics on Fano manifolds. ''Metric and Differential Geomtry'',
Proceeding of the 2008 conference celebrating J. Cheeger's 65th birthday, edited by Dai et al., Progress in Mathematics, volume 239. Birkh\"auser, 2012.

\bibitem[Ti12]{tian12}
Tian, G.: Partial $C^0$-estimates for K\"ahler-Einstein metrics. Communications in Mathematics and Statistics, {\bf 1}, no. 2 (2013), 105-113.

\bibitem[Ti13]{tian13}
Tian, G.: Stability of pairs. Preprint, arXiv:1310.5544.



\end{thebibliography}
\end{document}